\documentclass[11pt,a4paper,oneside,onecolumn]{amsart}
\usepackage{mbc}
\usepackage{amsmath}
\usepackage{amsfonts}
\usepackage{amssymb}
\usepackage{fp}
\usepackage{hyperref}
\DeclareMathOperator{\card}{Card}

\DeclareMathOperator{\con}{Con}
\DeclareMathOperator{\free}{free}
\newcommand {\disjcup} {\sqcup}
\newcommand {\disjbigcup} {\bigsqcup}
\newcommand{\nor}{\downarrow}
\newcommand{\N}{\mathbb{N}}

\newcommand{\Z}{\mathbb{Z}}
\newcommand{\extprod}[1][\cdot]{#1}

\newcommand{\cartprod}{\times}

\newcommand{\rem}[1]{}

\newcommand{\true}{\top}
\newcommand{\false}{\perp}
\newcommand{\truthvalueset}{\left\{ \true, \false \right\}}
\newcommand{\st}{|}
\newcommand{\powset}[1]{2^{#1}}
\newcommand{\informal}[1]{#1}
\newcommand{\proves}{\vdash}
\theoremstyle{plain}

\title[Yet another proof of G\"odel's completeness theorem]{Yet another proof of G\"odel's completeness theorem for first-order classical logic}
\date{}

\thanks{}

        
\dedicatory{}

\begin{document}

\maketitle

\begin{abstract}
A Henkin-style proof of completeness of first-order classical logic is given with respect to a very small set (notably missing cut rule) of Genzten deduction rules for intuitionistic sequents.
\comment{\\
The attention is focused on using only effectively needed rules for each step in the proof, so as to be able to discern between the requirements of results usually taken for equivalent.} 
\\
Insisting on sparing on derivation rules, satisfiability theorem is seen to need weaker assumptions than completeness theorem, the missing request being exactly the rule $ \lnot \lnot \varphi \rightarrow \varphi$, which gives a hint of intuitionism's motivations from a classical point of view.
\\
A bare treatment of standard, basic first-order syntax somehow more algebraic-flavoured than usual is also given.
\end{abstract}

\section{Introduction}
Some twenty years after G\"odel gave the first proof of its completeness result (see \cite{van1967frege} for an english translation of the original paper), Henkin (\cite{henkin1949completeness}) found a substantially different strategy to prove it, giving an alternative approach which has become standard in the exposition of the subject, see for example \cite{ebbinghaus1994mathematical}.

The latter treatment inspired the present one. 
In particular the proof mechanism of \ref{TeoSatisfiability} is unchanged. 
\\
What has been changed is the derivation of \ref{PropSatisfiability}, which is the main result; it is given in a quite different way, attentive to the economy of derivation rules. 
As a consequence, it is shown that satisfiability theorem, which is the key for completeness in all the arguments known to author (G\"odel's included), requires one rule less than completeness.
\\
In other words, the rule set large enough to prove any true statement has one rule more than the rule set large enough that it can be enlarged without making any consistent set of formula inconsistent: compare \ref{PropCompletenessWRTConsistency} with \ref{TheoCompleteness}.  
\\
And anyway, the completeness theorem can inherently be proven according to a set of basic deduction rules (listed in \ref{TenRules}) which lacks ``Proof by cases'' rule, and then cut-rule, which is derived from it (see \cite{ebbinghaus1994mathematical}, IV.3.2), right from the start.
\\
As a byproduct, this led to a careful reformulation of consistency-related definitions, and to a generalization of the construction of Henkin for the model of a given consistent set of formula to the generic free interpretation and quotient construction (\eqref{DefnFreeInterpretation} and \eqref{QuotientInterpretation}).
\\
Further, a reformulation of first-order logic syntax and semantics in term of semigroups is given in \ref{SectSyntax} and \ref{SectSemantics}. 
A plus of this approach is that definitions and some proofs get shorter and compact, albeit arguably obscuring the intuitive meaning of things, which is what standard treatment is based on.
\\
All these addictional results are relegated to appendix so as not to obscure the main result.
\\
For yet alternative proofs of this same cornerstone result, the reader can see \cite{enderton1972mathematical} or \cite{smullyan1995first}.
\section{Sequents}
\lnote{Definire piu' precisamente cosa e' un sequente e cosa una dimostrazione} \rnote{A seconda di come lo fai, vedi se aggiungere weakening e contraction alle regole}
In sections \ref{SectSyntax} and \ref{SectSemantics}, syntax and semantics have been formalized. Here a formalization of deductive proofing is set up. All along the tractation of completeness theorem, we fix the symbols set $S$ and the variable symbols set $X$ and explicitly require them both to be at most countable.
\subsection{Basic derivation rules}
\label{TenRules}
\xdef\counter{0}
\FPeval{counter}{clip(\counter{} - 1)}
\xdef\counter{\FPprint\counter}
Let $\Gamma \subseteq F^S$ be finite.
\subsubsection{Assumption}
\label{AssRule}
\FPeval{counter}{clip(\counter{}) + 1}
\xdef\counter{\FPprint\counter}
\FPeval{AssRuleCounter}{clip(\counter{})}
\xdef\AssRuleCounter{\FPprint\AssRuleCounter}
Rule number \FPprint\AssRuleCounter
\begin{align*}
\begin{aligned}
\\
\hline
\Gamma &&  \varphi
\end{aligned}
&& \text{where} \left\{ \varphi \right\} \subseteq \Gamma
\end{align*}
\subsubsection{Antecedent}
\label{AntRule}
\FPeval{counter}{\counter{} + 1}
\xdef\counter{\FPprint\counter}
\FPeval{AntRuleCounter}{clip(\counter{})}
\xdef\AntRuleCounter{\FPprint\AntRuleCounter}
Rule number \FPprint\AntRuleCounter
\begin{align*}
\begin{aligned}
\Gamma && \varphi
\\
\hline
\Gamma' && \varphi
\end{aligned}
&& \text{where } \Gamma \subseteq \Gamma'
\end{align*}
\subsubsection{Reflexivity of Equality}
\label{ReflRule}
\FPeval{counter}{\counter{} + 1}
\xdef\counter{\FPprint\counter}
\FPeval{ReflRuleCounter}{clip(\counter{})}
\xdef\ReflRuleCounter{\FPprint\ReflRuleCounter}
Rule number \FPprint\ReflRuleCounter
\begin{align*}
\begin{aligned}
\\
\hline
\equiv t t
\end{aligned}
\end{align*}
\subsubsection{Substitution through Equality}
\label{SubstRule}
\FPeval{counter}{\counter{} + 1}
\xdef\counter{\FPprint\counter}
\FPeval{SubstRuleCounter}{clip(\counter{})}
\xdef\SubstRuleCounter{\FPprint\SubstRuleCounter}
Rule number \FPprint\SubstRuleCounter
\begin{align*}
\begin{aligned}
\Gamma && && \varphi\tfrac{t}{x}
\\
\hline
\Gamma && \equiv t t' && \varphi\tfrac{t'}{x}
\end{aligned}
\end{align*}
\subsubsection{Quantifier Introduction in succedent}
\label{ESRule}
\FPeval{counter}{\counter{} + 1}
\xdef\counter{\FPprint\counter}
\FPeval{ESRuleCounter}{clip(\counter{})}
\xdef\ESRuleCounter{\FPprint\ESRuleCounter}
Rule number \FPprint\ESRuleCounter
\begin{align*}
\begin{aligned}
\Gamma && \varphi \tfrac{t}{x}
\\
\hline
\Gamma &&  \exists x \varphi
\end{aligned}
\end{align*}
\subsubsection{Quantifier Introduction in antecedent}
\label{EARule}
\FPeval{counter}{\counter{} + 1}
\xdef\counter{\FPprint\counter}
\FPeval{EARuleCounter}{clip(\counter{})}
\xdef\EARuleCounter{\FPprint\EARuleCounter}
Rule number \FPprint\EARuleCounter
\begin{align*}
\begin{aligned}
\Gamma && \varphi \tfrac{x'}{x} && \psi
\\
\hline
\Gamma &&  \exists x \varphi && \psi
\end{aligned}
&&
\text{where } x' \in X \setminus \free \left( \Gamma \cup \left\{ \exists x \varphi, \psi \right\} \right)
\end{align*}
\subsubsection{NOR introduction}
\label{NorRule}
\FPeval{counter}{\counter{} + 1}
\xdef\counter{\FPprint\counter}
\FPeval{NorRuleCounter}{clip(\counter{})}
\xdef\NorRuleCounter{\FPprint\NorRuleCounter}
Rule number \FPprint\NorRuleCounter
\begin{align*}
\begin{aligned}
\Gamma && \nor \varphi_1 \varphi_1
\\
\Gamma && \nor \varphi_2 \varphi_2
\\
\hline
\Gamma &&  \nor \varphi_1 \varphi_2
\end{aligned}
\end{align*}
\subsubsection{NOR simmetry}
\label{NorCommRule}
\FPeval{counter}{\counter{} + 1}
\xdef\counter{\FPprint\counter}
\FPeval{NorCommRuleCounter}{clip(\counter{})}
\xdef\NorCommRuleCounter{\FPprint\NorCommRuleCounter}
Rule number \FPprint\NorCommRuleCounter
\begin{align*}
\begin{aligned}
\Gamma && \nor \varphi_1 \varphi_2
\\
\hline
\Gamma &&  \nor \varphi_2 \varphi_1
\end{aligned}
\end{align*}
\subsubsection{Contradiction +}
\label{PCtrRule}
\FPeval{counter}{\counter{} + 1}
\xdef\counter{\FPprint\counter}
\FPeval{PCtrRuleCounter}{clip(\counter{})}
\xdef\PCtrRuleCounter{\FPprint\PCtrRuleCounter}
Rule number \FPprint\PCtrRuleCounter
\begin{align*}
\begin{aligned}
\Gamma && \varphi && \psi
\\
\Gamma && \varphi && \nor \psi \phi
\\
\hline
\Gamma &&  && \nor \varphi \varphi
\end{aligned}
\end{align*}
\subsubsection{Contradiction -}
\label{NCtrRule}
\FPeval{counter}{\counter{} + 1}
\xdef\counter{\FPprint\counter}
\FPeval{NCtrRuleCounter}{clip(\counter{})}
\xdef\NCtrRuleCounter{\FPprint\NCtrRuleCounter}
Rule number \FPprint\NCtrRuleCounter
\begin{align*}
\begin{aligned}
\Gamma && \nor \varphi \varphi && \psi
\\
\Gamma && \nor \varphi \varphi && \nor \psi \phi
\\
\hline
\Gamma &&  && \varphi
\end{aligned}
\end{align*}
\subsection{Correctness. Compact notation for rules selection}
The rules given in \ref{TenRules} have a role similar to that of construction rules for $F^S$ and $T^S$ seen in \ref{TermsAndFormulas}: the former
decree which sequents are admitted in proofs, the latter ruled what strings of symbols are accepted to make computations on, and were built with syntactical conformity in mind, so to filter exactly the formulas that other authors call well-formed to emphasize the point.
Correspondingly, the first requisite one expects from derivation rules is that any formula derived from true assumption must be true as well, i.e. that basic derivation rules are correct. 
Checking that this is indeed the case for our basic rules is straightforward but tedious; let us just state a fundamental consequence:
\begin{teo}[Consistency theorem]
\label{TheoConsistency}
Any satisfiable set of formulas is consistent.
\end{teo}
Formal definition of consistency is to come in \ref{SectionConsistency}, anyway this should not pose a problem for some informal motivational discussion.
The proof of \ref{TheoConsistency} is much easier than that of its reverse (satisfiability theorem, \ref{TeoSatisfiability}), which is done below.
A first symptom of the trickiness involved: if one looks at \ref{TheoConsistency}'s statement, there's no ambiguity in it, in that the declared consistency holds irrespective of the deduction rules adopted; one can add or subtract arbitrary deduction rules and, provided they are correct, \ref{TheoConsistency} stands still.

On the other hand, if faced with statement ``Any consistent set of formulas is satisfiable'', one may well, and should, ask: ``Consistent according to what derivation rules?''.
\\
We shall later see, indeed,  that every basic rule (except \ref{NCtrRule} which is used only for the easy derivation, from \ref{TeoSatisfiability}, of the stronger completeness theorem, \ref{TheoCompleteness}) is needed along the proof, and this results in having to get fussy about always specifying according to which rules one is deriving.
\\
In particular, satisfiability theorem need all-around consistency for its thesis to hold, and is not granted to stand for weaker-consistent sets of formulas, simply because of the inability of finding some proof dispensing one of the basic rules.

So we need a compact way to always specify which rules are assumed at each stage of the work. The letter $D$ will denote the generic subset of the set of basic rules provided in \ref{TenRules}.
\\
Having enumerated each member of the latter permits to bijectively denote each member of its power set with a binary string of \FPeval{var}{clip(\counter{}+1)}\FPprint\var{} digits. That is, each subset of basic rules can be associated with a natural number less than{}\FPeval{var}{round(2^(\counter{}+1):0)}\FPprint\var{}.
\\
\FPeval{RuleChoice}{round(2^\AssRuleCounter + 2^\ReflRuleCounter + 2^\NorRuleCounter:0)}
For example the rule set 
\begin{align*}
D:= \left\{ \text{\nameref{ReflRule}} , \text{\nameref{NorRule}}
, \text{\nameref{AssRule}}
\right\}
\end{align*}
can be identified with the number \FPprint\RuleChoice.
\\
\FPeval{RuleChoice}{round(2^\AssRuleCounter + 2^\AntRuleCounter:0)}
With a handy notational abuse, one can thus write relations like $D \supseteq \FPprint\RuleChoice$ to compactly state that he is considering a set of rules $D$ including at least \nameref{AssRule} and \nameref{AntRule}.
\\
Finally, $\overline{D}$ will be the set of all the rules derivable given the starting rule set $D$.
\subsection{Some derivable rules}
\FPeval{RuleChoice}{round(2^AntRuleCounter + 2^AssRuleCounter + 2^PCtrRuleCounter :0)}
Here are some derivable rules for our convenience in future arguments.
Their derivation is omitted in most cases, being short and easy.
\\
If $D \supseteq \FPprint\RuleChoice$ then $\overline{D}$ contains the following rule:
\FPeval{RuleChoice}{round(2^AntRuleCounter + 2^AssRuleCounter + 2^PCtrRuleCounter + 2^NorCommRuleCounter :0)}
\begin{align*}
\begin{aligned}
\Gamma && \nor \varphi_1 \varphi_2
\\
\hline
\Gamma&&  \nor \varphi_1 \varphi_1
\end{aligned}
&& \text{ NOT introduction}
\end{align*}
If $D \supseteq \FPprint\RuleChoice$ then $\overline{D}$ contains the following rule:
\FPeval{RuleChoice}{round(2^AntRuleCounter + 2^AssRuleCounter + 2^NCtrRuleCounter :0)}
\begin{align*}
\begin{aligned}
\Gamma && \nor \varphi_1 \varphi_2
\\
\hline
\Gamma&&  \nor \varphi_2 \varphi_2
\end{aligned}
&& \text{ NOT introduction}
\end{align*}
If $D \supseteq \FPprint\RuleChoice$ then $\overline{D}$ contains the following rule:
\FPeval{var}{round( 2^AssRuleCounter + 2^AntRuleCounter  + 2^PCtrRuleCounter :0)}
\begin{align}
\label{NotNotRemove}
\begin{aligned}
\Gamma && \lnot \lnot \varphi
\\
\hline
\Gamma&&  \varphi
\end{aligned}
\end{align}
If $D \supseteq \FPprint\RuleChoice$ or $D \supseteq \FPprint\var$  then $\overline{D}$ contains the following rule:
\begin{align*}
\begin{aligned}
\Gamma &&  \varphi
\\
\hline
\Gamma&&  \lnot \lnot \varphi
\end{aligned}
\end{align*}
\begin{proof}
We do the case $D \supseteq \FPprint\RuleChoice$: the sequent $ \Gamma \varphi$ is given, and $\Gamma \lnot \lnot \lnot \varphi \lnot \lnot \lnot \varphi$ is introduced by \ref{AssRule}. Apply respectively \ref{AntRule} and \eqref{NotNotRemove} to trigger this derivation:
\begin{align*}
\begin{aligned}
\Gamma && \lnot \lnot \lnot \varphi && \varphi
\\
\Gamma && \lnot \lnot \lnot \varphi && \lnot \varphi
\\
\hline
\Gamma && && \lnot \lnot \varphi
\end{aligned}
&& \text{by } \eqref{NCtrRule}
\end{align*}
\end{proof}
\subsection{First properties of sequents. Consistency}
\label{SectionConsistency}
We continue to use $D$ to indicate a generic subset of basic rules, and study the scope of derivations when varying which rules we admit in $D$.
\begin{defn}
\label{DefExpansion}
The letter $D$ shall be overloaded to denote with $D \left( \Phi \right)$ the logical expansion of $\Phi$ according to $D$, that is the set of formulas
\begin{align*}
D\left( \Phi \right) := \left\{ \varphi \in F^S \st \Phi \proves_{D} \varphi \right\}
\end{align*}
$\Phi$ is said to be $D$-expanded iff $ D\left( \Phi \right) \subseteq \Phi$.
\end{defn}
Obviously $\overline{D} \left( \Phi \right) = D\left( \Phi \right)$, since derivable rules are just handy placeholders replaceable with a full derivation, similar to the concept of macro in computer science.

The following definitions characterize a set of formulas according to its logical properties, and complete the definitions given in \ref{SectFamiliesOfFormulas}, which pursued the same task in a sheer syntactical manner.
\begin{defn}
One says that $\Phi$ is $D$-consistent, and writes $\con_{D}{\Phi}$, iff $\forall \varphi \in F^{S}, \left\{ \psi \st \Phi \proves_{D} \psi \right\} \cap \left\{ \varphi, \nor \varphi \varphi \right\}$ has cardinality not exceeding $1$.
Equivalently, $D \left( \Phi \right)$ is not patently inconsistent.
\end{defn}
\begin{defn}
$\Phi$ is said to be $D$-maximal iff it is a $D$-consistent covering.
\end{defn}
In making use of above definitions we shall often elide the $D$ letters when there is no reasonable ambiguity possible.
\begin{lemma}
\label{LemmaMaximalCovering}
If $D$ includes assumption rule, then 
\begin{enumerate}
\item
$\Phi$ is maximal $\Leftrightarrow \Phi$ is a minimal covering and $\Phi$ is expanded.
\item
$\Phi$ is maximal $ \Rightarrow \left( \con \Phi \text{ and } \left( \con \Phi \cup \left\{ \varphi \right\} \Rightarrow \varphi \in \Phi \right) \right) $
\item
$\Phi \subseteq D\left( \Phi \right)$
\end{enumerate}
\end{lemma}
\begin{proof}
\mbox{}
\begin{enumerate}
\item
\begin{description}
\item[ $\Rightarrow$ direction:]
$\Phi$ being a covering implies $\card \left\{ \varphi , \nor \varphi \varphi \right\} \cap \Phi \geq 1 $, while it being consistent implies, using assumption rule, that same cardinality being $\leq 1$.
\\
Suppose $\Phi \proves \varphi$. By contradiction, $\nor \varphi \varphi$ cannot be in $\Phi$, for if it were then, by assumption rule, $\Gamma \proves \nor \varphi \varphi $, which would violate consistency. Thus $\nor \varphi \varphi \notin \Phi$ which, minimal covering property having already been ascertained, implies $\varphi \in \Phi$.
\item[ $\Leftarrow$ direction:]
It all reduces to showing that $\Phi$ is consistent:
\begin{align*}
\Phi \proves \varphi \Rightarrow \varphi \in \Phi \Rightarrow \nor \varphi \varphi \notin \Phi \Rightarrow \Phi \not{\proves} \nor \varphi \varphi,
\end{align*}
where middle deduction uses minimal covering hypothesis, while external ones both use expansion hypothesis. No deduction rule was needed, which is remarkable.
\end{description}
\item
Assuming LHS of last thesis, consistency and assumption rule give $ \nor \varphi \varphi \notin \Phi \cup \left\{ \varphi \right\} \Rightarrow \nor \varphi \varphi \notin \Phi$. 
So $\varphi \in \Phi$, being the latter a minimal covering.
\item
Obvious.
\end{enumerate}
\end{proof}

\begin{lemma}
\label{ConsistencyVSImplication}
\FPeval{RuleChoice}{round(2^\PCtrRuleCounter + 2^\AntRuleCounter:0)}
$D \supseteq \FPprint\RuleChoice$, $\Phi \subseteq F^{S}$ and $\varphi \in F^{S}$. Then
\begin{enumerate}
\item
Not $\con \Phi \cup \left\{ \varphi \right\} \Rightarrow \Phi \proves \nor \varphi \varphi$
\item
$\Phi \proves \varphi$ and $\con \Phi \Rightarrow \con \Phi \cup \left\{  \varphi \right\}$
\end{enumerate}
\end{lemma}
\begin{proof}
\mbox{}
\begin{enumerate}
\item
There are $\psi \in F^{S}, \ \Gamma' \subseteq \Phi \cup \left\{ \varphi \right\}$, $\Gamma'$ finite such that $ \proves \Gamma' \lnot \psi $ and $ \proves \Gamma'  \psi $.
\\
This can be restated as  $ \proves \Gamma   \varphi \psi $ and $ \proves \Gamma  \varphi \nor \psi \psi $, with $\Gamma \subseteq \Gamma' \cap \Phi$; by \ref{AntRule} $\Gamma$ can be made non-empty, so using \ref{PCtrRule}: 
\begin{align*}
\begin{aligned}
\Gamma &&  \varphi && \psi
\\
\Gamma && \varphi &&  \nor \psi \psi
\\
\hline
\Gamma && &&   \nor \varphi \varphi
\end{aligned}
\end{align*}
yielding $ \proves \Gamma \nor \varphi \varphi $, i.e. $ \Phi \proves \nor  \varphi  \varphi $.
\item
The hypothesis implies $\Phi \not \proves \nor \varphi \varphi $ which, using previous point, gives $\con \Phi \cup \left\{ \varphi \right\}$.
\end{enumerate}
\end{proof}

\FPeval{RuleChoice}{round( 2^AntRuleCounter + 2^PCtrRuleCounter :0)}
\FPeval{var}{round( 2^AntRuleCounter + 2^NCtrRuleCounter :0)}
\begin{oss}
\label{WeakConsistency}
If $D \supseteq \FPprint\RuleChoice$ or  $D \supseteq \FPprint\var$, then
\begin{align*}
\con_D \Phi \Leftrightarrow \text{ there is } \varphi \in F^S \st \Phi \not \proves_D \lnot \varphi 
\end{align*}
\end{oss}
The easy proof is omitted.
\begin{defn}[Maximization constructions]
\label{DefMaximizationConstructions}
Given a set of rules $D$ and an enumeration $\alpha$ of $F^S$ (so that $F^S= \left\{ \varphi_1, \varphi_2, \ldots \right\}$, where $\varphi_j :=\alpha\left( j \right)$), define, for any $\Phi \subseteq F^S$: 
\begin{align}
\Phi_0 &:= \Phi & \Phi'_{0} &:= \Phi,
\nonumber
\\
&\text{ and then recursively for $j=1,2,3, \ldots$}
\nonumber
\\
\label{Maximization}
\Phi_j &:= 
\begin{cases}
\Phi_{j-1} \cup \left\{ \varphi_j 
\right\}
& \text{ if } \Phi_{j-1} \not \proves_{D}   \lnot \varphi_{j}
\\
\Phi_{j-1} \cup \left\{ \lnot \varphi_j 
\right\}
& \text{ if } \Phi_{j-1} \proves_{D}  \lnot \varphi_{j}
\end{cases}
&
\Phi'_j &:= 
\begin{cases}
\Phi'_{j-1} \cup \left\{ \lnot \lnot \varphi_j 
\right\}
& \text{ if } \Phi'_{j-1} \not \proves_{D}   \lnot \varphi_{j}
\\
\Phi'_{j-1} \cup \left\{ \lnot \varphi_j 
\right\}
& \text{ if } \Phi'_{j-1} \proves_{D}  \lnot \varphi_{j}
\end{cases}.
\end{align}
Finally, set $\mathcal{M}^+_D\left( \Phi \right):= \bigcup_{j=1}^{+\infty} \Phi_j$ and  $\mathcal{M}^-_D\left( \Phi \right):= \bigcup_{j=1}^{+\infty} \Phi'_j$.
\end{defn}
\begin{oss}
\label{MPlusCovering}
\mbox{}
\begin{itemize}
\item
In general, $\mathcal{M}^{\pm}_D\left( \Phi \right)$ depends also on the choice of $\alpha$. However, since we are not going to rely on this fact, we can go for the lighter notation.
\item
$\mathcal{M}^{\pm}_D\left( \Phi \right) \supseteq \Phi$: $\mathcal{M}^{+}_D\left( \Phi \right)$ and $\mathcal{M}^{-}_D\left( \Phi \right)$ are both extensions of $\Phi$.
\item
$\mathcal{M}^+_D\left( \Phi \right)$ is obviously a minimal covering.
\end{itemize}
\end{oss}
\begin{oss}
\label{MPlusConsistency}
\FPeval{RuleChoice}{round(2^\PCtrRuleCounter + 2^\AntRuleCounter :0)}
$D \supseteq \FPprint\RuleChoice , \con_D \Phi \Rightarrow 
\con_D \mathcal{M}^+_{D} \left( \Phi \right)$. 
\end{oss}

\begin{proof}
Referring to \ref{DefMaximizationConstructions}, each $\Phi_{j}$ is recursively $D$-consistent: $\con \Phi_{j-1} \Rightarrow \con \Phi_{j}$, as seen by applying first thesis of \ref{ConsistencyVSImplication} to first branch of \eqref{Maximization} and second to second.
Since the $\Phi_{j}$'s are concentric, any finite subset of $\mathcal{M}^+_D \left( \Phi \right) $ is included in some $\Phi_j$ and thus consistent.
Thus $\mathcal{M}^+_D \left( \Phi \right) $ is consistent.
\end{proof}

\subsection{The Henkin model}
Now a construction procedure for a model of a given sets of formulas $\Phi$ is exhibited. Care is taken upon always specifying which derivation rules we will be using from time to time; this issue is closely related to the study of the minimal requirements $\Phi$ must obey in order to be able to apply such procedure.
\\
The process is in three steps: first $\Phi$ is enlarged to a somehow maximal family, then the latter's good properties are exploited to build the needed interpretation, and finally to show that this interpretation is a model for the derived family, and thus also for the starting subfamily.
\subsubsection{Building an interpretation}
\label{SectBuildingAnInterpretation}
Fix a formal structure $S$. The free interpretation $\left( T^{S}, \omega_{S} \right)$ (cfr. \ref{FreeInterprExt}) can always be built, the matter is how to extend it to negative-arity symbols of $S$.
This can obviously, as discussed in \ref{FreeInterprExt}, be done anyhow one wants, but in the present situation we have $\Phi$ to guide us in the task in a natural way: set
\begin{align*}
f \colon \#^{-1}\left( \Z^{-} \right) \ni s \mapsto
\left(
\left( \nple{t}{-\#\left( s \right)} \right)
\mapsto
\begin{cases}
\true & \text{ if } s \circ t_1 \circ \ldots \circ t_{- \# \left( s \right) } \in \Phi
\\
\false & \text{ else }
\end{cases}
\right)
\end{align*}
and call  $\mathcal{I}_{\Phi} := \left( T^{S}, \omega_{S}^f \right)$ (cfr \ref{FreeInterprExt}).
Note that $\mathcal{I}_{\Phi}$ satisfies (is a model of)
\begin{align*}
\Phi \cap
\bigcup_{s \in {\#}^{-1} \left( \Z^- \right)} 
\circ \left( \left\{ s \right\} \cartprod {\left( T^S \right)}^{\left| \# \left( s \right) \right|} \right)
= \Phi \cap \left\{ \varphi \in F^S_0 \st \ \varphi(1 :1) \notin \left\{ \equiv \right\} \right\}
\end{align*}
while the atomic formulas of $\Phi$ starting with $\equiv$ are satisfied if and only if they are literal identities, i.e. of the form $\equiv t t$ for some $t$ in $T^{S}$. At this time, nothing can be said about higher-order formulas, because no request on $\Phi$ has been done yet. Now we fix one problem a time, starting from basics and trying to accomodate  into the model given by $\Phi$ first the atomic formulas with $\equiv$ as a first character. This first step alone will actually get most of the work done.
\\
The idea is to apply the quotient construction introduced in \eqref{QuotientInterpretation} according to the natural equivalence relation $\Phi$ can give as soon as it is rich enough in formulas. 
\\
That is, define the binary relation $E_{\left( \Phi, D \right)}$ on $T^{S}$ as
\begin{align*}
t_{1} E_{\left( \Phi, D \right)} t_{2} \Leftrightarrow \Phi \proves_{D} \equiv t_{1} t_{2},
\end{align*}
where $D$ is a finite set of deduction rules.

\begin{lemma}
\label{LemmaFreeInterpretationConstruction}
\mbox{}
\begin{enumerate}
\item
\FPeval{RuleChoice}{round(2^\AssRuleCounter + 2^\AntRuleCounter + 2^\ReflRuleCounter + 2^\SubstRuleCounter :0)}
If $D$ includes 
\ref{AssRule}, \ref{AntRule}, \ref{ReflRule}, \ref{SubstRule}
rules (that is, $D \supseteq \FPprint\RuleChoice$), then $E_{\left( \Phi, D \right)}$ is an equivalence relation over $T^{S}$.
\item
If, moreover, $\Phi$ is \mbox{$D$-expanded} (see \ref{DefExpansion}), $E_{\left( \Phi, D \right)}$ is preserved through $\omega_{\left( S, \Phi \right)}$ (cfr. \eqref{DefPreservation} ).
\end{enumerate}
\end{lemma}
\begin{proof}
\mbox{}
\begin{enumerate}
\item
The three properties of reflexivity, symmetry and transitivity to be shown to establish first part of the thesis correspond to show the derivability of the respective sequents:
\begin{align*}
&\equiv t_1 t_1
\\
&\equiv t_1 t_2 \equiv t_2 t_1
\\
&\equiv t_1 t_2 \equiv t_2 t_3 \equiv t_1 t_3
\end{align*}
The first sequent derives directly from a rule assumed to belong to $D$.
\\
The second one can be derived using just the two rules: reflexivity rule gives the sequent $\equiv t_1 t_1$ which can be seen as $\left[ \equiv x_1 t_1 \right]\frac{t_1}{x_1}$ (here $x_1$ is supposed not to occur in $t_1$). Antecedent rule then produces the sequent $\equiv t_1 t_2 \left[ \equiv x_1 t_1 \right]\frac{t_1}{x_1} $, which yields, through substitution rule, $\equiv t_1 t_2 \left[ \equiv x_1 t_1 \right]\frac{t_2}{x_1}${}$=\equiv t_1 t_2  \equiv t_2 t_1$.
\\
For the last sequent, start deriving the sequent $ \equiv t_1 t_2 \equiv t_2 t_3 \equiv t_1 t_2${}$= \equiv t_1 t_2 \equiv t_2 t_3 [\equiv t_1 x_1]\frac{t_2}{x_1} $ (assuming $x_1$ not occurring in $t_1$) via assumption rule, and then $ \equiv t_1 t_2 \equiv t_2 t_3 \equiv t_1 t_3$ via substitution rule.
\item
Let $s \in \#^{-1} \left( \Z \setminus \left\{ 0 \right\} \right)$ and $\nple {t}{\left| \#\left( s \right)  \right| }$, $\nple {t'}{\left| \#\left( s \right)  \right| } \in T^{S}$ be given such that $\Phi \proves_{D} \equiv t_{j} t_{j}'$. By expansion hypothesis, we can strengthen this to $ \equiv t_{j} t_{j}' \in \Phi $. As stated in \eqref{DefPreservation}, we must show that
\begin{align}
\label{QuotientInterpretationCondition}
\omega_{\left( S, \Phi \right)} \left( s \right)  \left( \left( \nple{t}{\left| \#\left( s \right) \right|} \right)  \right)
{E_{\left( \Phi, D \right)}^{\prime}}
\omega_{\left( S, \Phi \right)} \left( s \right)  \left( \left( \nple{t'}{\left| \#\left( s \right) \right|} \right)  \right)
\end{align}
Let us go by cases.
\begin{description}
\item[$ \# \left( s \right) > 0 $:]
Condition \eqref{QuotientInterpretationCondition} becomes
\begin{align*}
\Phi \proves_{D} \equiv 
s \circ t_1 \circ \ldots \circ t_{\#\left( s \right)}
s \circ t_1' \circ \ldots \circ t_{\#\left( s \right)}'
\end{align*}
$ \proves_{D}  \equiv 
s \circ t_1 \circ \ldots \circ t_{\#\left( s \right)}
s \circ t_1 \circ \ldots \circ t_{\#\left( s \right)}$ 
by reflexivity of equivalence rule of $D$. From this sequent we obtain the sequent 
$ \equiv t_1 t_1' \ldots \equiv t_{\#(s)} t_{\#(s)}' \equiv 
s \circ t_1 \circ \ldots \circ t_{\#\left( s \right)}
s \circ t_1 \circ \ldots \circ t_{\#\left( s \right)}  
= \equiv t_1 t_1' \ldots \equiv t_{\#(s)} t_{\#(s)}' \left[ \equiv 
s \circ t_1 \circ \ldots \circ t_{\#\left( s \right)}
 s \circ x_1 \circ \ldots \circ t_{\#\left( s \right)} \right] \frac{t_1}{x_1}  
$ by antecedent rule of $D$, and then
\begin{align*}
\equiv t_1 t_1' \ldots \equiv t_{\#(s)} t_{\#(s)}' \left[ \equiv 
s \circ t_1 \circ \ldots \circ t_{\#\left( s \right)}
 s \circ x_1 \circ \ldots \circ t_{\#\left( s \right)} \right] \frac{t_1'}{x_1}
\end{align*}
by substitution rule of $D$. Here, always using variable $x_1$ which is assumed not to occur in any $t_j$ or $t_j'$, we iterate this procedure to obtain the sequent
\begin{align*}
\equiv t_1 t_1' \ldots \equiv t_{\#(s)} t_{\#(s)}' \equiv 
s \circ t_1 \circ \ldots \circ t_{\#\left( s \right)}
 s \circ t_1' \circ \ldots \circ t_{\#\left( s \right)}' 
\end{align*}
Now, since $\Phi$ is expanded, all the formulas $\equiv t_j t_j'$ belong to $\Phi$, so the last sequent implies that $\Phi \proves_D \equiv 
s \circ t_1 \circ \ldots \circ t_{\#\left( s \right)}
 s \circ t_1' \circ \ldots \circ t_{\#\left( s \right)}'$.

\item[$ \# \left( s \right) < 0 $:]
Condition \eqref{QuotientInterpretationCondition} becomes
\begin{align*}
\omega_{\left( S, \Phi \right)} \left( s \right) \left( \left( \nple {t}{-\#\left( s \right)} \right) \right)
=
\omega_{\left( S, \Phi \right)} \left( s \right) \left( \left( \nple {t^{\prime}}{-\#\left( s \right)} \right) \right)
\end{align*}
To make argument easier, we split the it into a further pair of subcases.
\begin{description}
\item[$s \circ t_1 \circ \ldots \circ t_{-\#\left( s \right)} \in \Phi $]
\mbox{}\\
We must prove that $s \circ t_{1}^{\prime} \circ \ldots \circ t_{-\#\left( s \right)}^{\prime} \in \Phi $.
\\
Starting sequent 
$s \circ t_{1}^{} \circ \ldots \circ t_{-\#\left( s \right)}^{} \  s \circ t_{1}^{} \circ \ldots \circ t_{-\#\left( s \right)}^{} $ is obtained using assumption rule. Then by antecedence: $s \circ t_{1}^{} \circ \ldots \circ t_{-\#\left( s \right)}^{} \ \equiv t_1 t'_1 \ s \circ t_{1}^{} \circ \ldots \circ t_{-\#\left( s \right)}^{} $ and finally by substitution, as already done in this proof:  $s \circ t_{1}^{} \circ \ldots \circ t_{-\#\left( s \right)}^{} \ \equiv t_1 t'_1 \ s \circ t_{1}^{\prime} \circ \ldots \circ t_{-\#\left( s \right)}^{} $. Iterating antecedence/substitution we get our ultimate sequent
$s \circ t_{1}^{} \circ \ldots \circ t_{-\#\left( s \right)}^{} \ \equiv t_1 t'_1 \ldots \equiv t_{-\#\left( s \right)}  t'_{-\#\left( s \right)} \ s \circ t_{1}^{\prime} \circ \ldots \circ t_{-\#\left( s \right)}^{\prime}$, which says $\Phi \proves_{D} s \circ t_{1}^{\prime} \circ \ldots \circ t_{-\#\left( s \right)}^{\prime}$ since $s \circ t_{1}^{} \circ \ldots \circ t_{-\#\left( s \right)}^{} \ \equiv t_1 t'_1 \ldots \equiv t_{-\#\left( s \right)}  t'_{-\#\left( s \right)}$ are all formulas of $\Phi$.
\item[$s \circ t_1 \circ \ldots \circ t_{-\#\left( s \right)} \notin \Phi $]
\mbox{}
\\
Suppose that $s \circ t'_1 \circ \ldots \circ t'_{-\#\left( s \right)} \in \Phi $. Since $\Phi$ contains also swapped equivalences $\equiv t'_1 t_1 \ldots \equiv t_{ - \#\left( s \right)}  t_{ - \#\left( s \right)}'$ (this can be shown using just reflexivity, antecedent and substitution rules), the same last argument could be repeated to show $\Phi \proves_{D} s \circ t_{1} \circ \ldots \circ t_{-\#\left( s \right)}$ and thus (by expansion hypothesis) $s \circ t_{1} \circ \ldots \circ t_{-\#\left( s \right)} \in \Phi$, a contradiction.  
\end{description}
\end{description}
\end{enumerate}
\end{proof}
Ultimately, we have cascaded the three constructions of \ref{StandardConstructions}: first free interpretation, then extension of free interpretation with a $f$ derived from $\Phi$, and finally (with \ref{LemmaFreeInterpretationConstruction} stating the constraints on $D$ and the bounds between $D$ and $\Phi$ needed to take this last step) quotient interpretation, to get the following definition.
\FPeval{RuleChoice}{round(2^\AssRuleCounter + 2^\AntRuleCounter + 2^\ReflRuleCounter + 2^\SubstRuleCounter :0)} 
\begin{defn}[Henkin model]
If $D \supseteq \FPprint\RuleChoice$ and $\Phi$ is any $D$-expanded set of formulas of $F^S$, set
\begin{align*}
\mathcal{I}_{\left( \Phi, D \right)} 
:= \frac{\mathcal{I}_{\Phi}}{E_{\left( \Phi , D \right)}}
\end{align*}
\end{defn}
This interpretation is the tool needed to work out satisfiability, and thus completeness, results in the sequel.
\subsubsection{Making it a model}
\label{SectMakingItAModel}
\begin{lemma}
\label{AtomicFormulasCompleteness}
\FPeval{RuleChoice}{round(2^\AssRuleCounter + 2^\AntRuleCounter + 2^\ReflRuleCounter + 2^\SubstRuleCounter:0)}
$D \supseteq \FPprint\RuleChoice$, $\Phi$ $D$-expanded, $\varphi \in F_{0}^{S}$. Then
\begin{align*}
 \mathcal{I}_{\left( \Phi, D \right)} \models \varphi \Leftrightarrow \varphi \in \Phi 
\end{align*}
\end{lemma}
\begin{proof}
By cases according to the first character of $\varphi$.
\begin{description}
\item[Case $\varphi = \equiv t_{1} t_{2}$]
\begin{align*}
\mathcal{I} \models \varphi &\Leftrightarrow \mathcal{I}\left( t_1 \right) =  \mathcal{I}\left( t_2 \right) \Leftrightarrow \mathcal{I}_{\Phi} \left( t_1 \right) E_{\left( \Phi, D \right)} \mathcal{I}_{\Phi} \left( t_2 \right)
\\
&\Leftrightarrow t_1 E_{\left( \Phi, D \right)} t_2 \Leftrightarrow \Phi \proves_{D} \equiv t_1 t_2 \Leftrightarrow \varphi \in \Phi 
\end{align*}
\item[Case $\varphi = s t_{1} \ldots t_{- \# \left( s \right)}$ for some $s \in \#^{-1} \left( \Z^{-} \right)$]
\begin{gather*}
\mathcal{I} \models \varphi 
\Leftrightarrow \omega_{\left( S, \Phi \right)} \left( s \right) \left( \mathcal{I}_{\Phi} \left( t_1 \right), \ldots, \mathcal{I}_{\Phi} \left( t_{-\#\left( s \right)} \right) \right) = \true
\\
\Leftrightarrow \omega_{\left( S, \Phi \right)} \left( s \right) \left(  t_1 , \ldots,  t_{-\# \left( s \right)} \right) = \true
\Leftrightarrow s t_1 \ldots t_{-\#\left( s \right)} \in \Phi
\end{gather*}
\end{description}
\end{proof}
Note how consistency is not at all involved in \ref{AtomicFormulasCompleteness}'s statement and proof. That is a weak result, meaning that it works for any kind of $\Phi$. Of course, if $\Phi$ is inconsistent, $\mathcal{I}_{\left( \Phi, D \right)}$ will behave inconsistently, too:
\begin{oss}[see  \cite{katz1994contemporary} for definitions] 
Given a formula $\varphi$, one can consider the set of the minterms occurring in each of its prime implicants, and thus obtain a family of sets of the form $\left\{ \varphi_1, \ldots, \varphi_n \right\}$, with each $\varphi_j$ in $F^S_0 
\disjcup \lnot\left( F^S_0 \right)$. It is easy to check that $\mathcal{I}_{\left( \Phi, D \right)} \models \varphi $ if and only if at least one set of this family is included in $\Phi$.
\end{oss}
Previous remark hints at the fact that $\Phi$ must itself satisfy some additional ``completeness'' request if one wants $\mathcal{I}_{\left( \Phi, D \right)}$ to be an interpretation for the \emph{whole} $\Phi$. That is why \ref{LemmaFreeInterpretationModel} makes the request of minimal covering.
\\
In a similar way, one has to request witness-completeness for $\Phi$ in the following precise sense.
\begin{defn}
$\Phi$ is witness-furnished iff $\exists x_m \varphi \in \Phi \Rightarrow$ there is $t \in T^{S} \st \ \varphi \frac{t}{x_m} \in \Phi$.
\end{defn}
One might say that to give an interpretation of non-atomic formulas, $\Phi$ must be in some sense complete according to both $\exists$ and $\nor$, which are the only two possible starting characters of a non-atomic formulas.
\\
However, $\exists$-completeness can be consider of a more technical, and less fundamental, nature than $\nor$-completeness. Indeed, the former will be shown to be dispensable in final result, and only needed to perform the proof.
What's more, if one wants to extend present result to the case of uncountable symbol set $S$ (not treated here), he finds that is $\nor$ completion, not the need of witnesses, which poses the hardest challenges and forces resorting to the axiom of choice to do what here is done in \ref{DefMaximizationConstructions} relying on countability of $S$. See \cite{ebbinghaus1994mathematical}. 
\begin{lemma}
\label{LemmaFreeInterpretationModel}
\FPeval{RuleChoice}{round(2^\AssRuleCounter + 2^\AntRuleCounter + 2^\ReflRuleCounter + 2^\SubstRuleCounter + 2^\NorRuleCounter + 2^\NorCommRuleCounter + 2^\ESRuleCounter:0)}
$D \supseteq \FPprint\RuleChoice$, $\Phi$ $D$-expanded, $\Phi$ minimal covering, $\Phi$ witness-furnished, $\varphi \in F^S$. Then
\begin{align*}
\mathcal{I}_{\left( \Phi, D \right)} \models \varphi \Leftrightarrow \varphi \in \Phi
\end{align*}
\end{lemma}
\begin{proof}
By induction on the depth $n$ of $\varphi$ (as defined in \ref{DefDepth}). Triggering case $n=0$ is granted by \ref{AtomicFormulasCompleteness}. Taking for proven the thesis for formulas of depth $\leq n$, consider a formula $\varphi$ of depth $n+1$. By cases:
\begin{description}
\item[$\varphi$ is of the form $\nor \varphi_1 \varphi_2$:]
\begin{align*}
\mathcal{I} \models \nor \varphi_1 \varphi_2 & \Leftrightarrow \mathcal{I} \not \models \varphi_1 \text{ and } \mathcal{I} \not \models \varphi_2 \Leftrightarrow
\\
\left\{ \varphi_1, \varphi_2 \right\} \cap \Phi = \emptyset & \overset{* }{\Leftrightarrow} \Phi \proves \nor \varphi_1 \varphi_1 \text{ and } \Phi \proves \nor \varphi_2 \varphi_2
\\
& \Leftrightarrow \Phi \proves \nor \varphi_1 \varphi_2,
\end{align*}
last equivalence being given by NOR and NOT introduction rules, one per direction, and marked equivalence being the one requiring minimal covering hypothesis. The remaining steps go by definition and by inductive hypothesis.
\item[$\varphi$ is of the form $\exists x_m \varphi$:]
\begin{gather*}
\mathcal{I} \models \exists x_m \varphi \Leftrightarrow \mathcal{I} \frac{[t]}{x_m} \models \varphi \text{ for some } t \in T^{S} 
\\
\Leftrightarrow
\mathcal{I} \models \varphi \frac{t}{x_m} \text{ for some } t \in T^{S} \Leftrightarrow  \varphi \frac{t}{x_m}  \in \Phi \text{ for some } t \in T^{S}
\\
\Leftrightarrow \exists x_m \varphi \in \Phi
\end{gather*}
Last equivalence stands on application of \ref{ESRule} followed by use of expansion hypothesis ($ \Rightarrow$) and on assumption that $\Phi$ is witness-furnished ($\Leftarrow$).
Here $\left[ t \right]$ is the equivalence class of the term $t$ according to $E_{\left( \Phi, D \right)}$.
As a passing note, the $t$ occurring in previous equivalence chain is always the same, although this accident is not exploited.
\end{description}
\end{proof}
\begin{lemma}
\label{SingleWitnessInclusion}
\FPeval{RuleChoice}{round(2^AntRuleCounter + 2^ReflRuleCounter + 2^EARuleCounter +2^ESRuleCounter + 2^PCtrRuleCounter:0)}
\FPeval{var}{round(2^AntRuleCounter + 2^ReflRuleCounter + 2^EARuleCounter +2^ESRuleCounter + 2^NCtrRuleCounter:0)}
$D \supseteq \FPprint\RuleChoice$ or $D \supseteq \FPprint\var$; given $\Phi \cup \left\{ \varphi \right\} \subseteq F^S$, $x_m \in X$:
\begin{align*}
\con \Phi \cup {\exists x_m \varphi} &\Rightarrow \con \Phi \cup \left\{ \varphi \frac{x_n}{x_m} \right\} & \forall x_n \in X \setminus \free\left( \Phi \cup \left\{ \exists x_m \varphi \right\} \right)
\end{align*}
\end{lemma}
\begin{proof}
Name 
$\Phi_1 := \Phi \cup {\exists x_m \varphi}; \ 
\Phi_2 := \Phi \cup \left\{ \varphi \frac{x_n}{x_m} \right\}$ to ease things up.
\\
The following derivation:
\begin{align*}
\begin{aligned}
\\
\hline
\equiv  x_1 x_1 
\end{aligned}
&
\eqref{ReflRule}
&
\begin{aligned}
&\equiv  x_1 x_1
\\
\hline
\exists x_1 &\equiv  x_1 x_1
\end{aligned}
& \eqref{ESRule}
\end{align*}
says $\emptyset \proves_D \exists x_1 \equiv  x_1 x_1 $, and thus $\Phi_1 \proves_D \exists x_1 \equiv  x_1 x_1 $ and $\Phi_2 \proves_D \exists x_1 \equiv  x_1 x_1 $ by \eqref{AntRule}.
By contradiction now the negation of the latter formula is exhibited to be derivable from $\Phi_1$, which on the other hand is granted to be consistent.
\\
Suppose indeed that $\Phi_2 \proves_D \lnot \exists x_1 \equiv  x_1 x_1 $.
\\
Then there is a sequent $\Gamma \ \varphi \frac{x_n}{x_m} \ \lnot \exists x_1 \equiv  x_1 x_1 $, with $\Gamma$ finite subset of $\Phi$.\\
Now, since $x_n \in X \setminus \free\left( \Phi \cup \left\{ \exists x_m \varphi \right\} \right)$, $x_n$ does not occur free in $\Gamma \ \exists x_m \varphi \ \lnot \exists x_1 \equiv  x_1 x_1 $, permitting
\begin{align*}
\begin{aligned}
\Gamma && \varphi \tfrac{x_n}{x_m} && \lnot \exists x_1 \equiv  x_1 x_1 
\\
\hline
\Gamma && \exists x_m \varphi && \lnot \exists x_1 \equiv  x_1 x_1 
\end{aligned}
&& 
\eqref{EARule},
\end{align*}
so that $\Phi_1 \proves_D \lnot \exists x_1 \equiv x_1 x_1$.
\\
In the end, it must be  $\Phi_2 \not{\proves_D} \lnot \exists x_1 \equiv  x_1 x_1 $. This, using lastly \eqref{WeakConsistency} (which in turn employs a contradiction rule), equates to $\Phi_2$ being consistent.
\end{proof}

\begin{lemma}[Witness Adjunction Construction]
\label{LemmaWitnessInclusion}
\FPeval{RuleChoice}{round(2^AntRuleCounter + 2^ReflRuleCounter + 2^EARuleCounter +2^ESRuleCounter + 2^PCtrRuleCounter:0)}
\FPeval{var}{round(2^AntRuleCounter + 2^ReflRuleCounter + 2^EARuleCounter +2^ESRuleCounter + 2^NCtrRuleCounter:0)}
$D \supseteq \FPprint\RuleChoice$ or $D \supseteq \FPprint\var$; if $\Phi \subseteq F^S$ is $D$-consistent and $ X \setminus \free\left( \Phi \right)$ is countable, there is $\mathcal{W}_{D} \left( \Phi \right) \subseteq F^S, \ \mathcal{W}_{D} \left( \Phi \right) \supseteq \Phi $ such that
\begin{itemize}
\item
$\mathcal{W}_{D} \left( \Phi \right) $ is consistent and witness-furnished.
\item
Every $D$-consistent superset of $\mathcal{W}_{D} \left( \Phi \right)$ is witness-furnished.
\end{itemize}
\end{lemma}
\begin{proof}
The set of formulas starting with $\exists$ is countable with its superset $F^S$, and can thus be written as $\disjbigcup_{j \in \Z^+} \left\{ \exists x_{\alpha \left( j \right)} \varphi_{\beta \left( j \right)} \right\}$.
\\
Define 
$
\Phi_{0} := \Phi
$,
and then recursively for $j=1,2,3, \ldots$
\begin{align}
\label{FirstNonFreeVar}
k_j:= \min \left\{ 
l \in \Z^{+} \st \ x_{l} \in  X \setminus \free \left( \Phi_{j-1} \cup
\left\{ \exists x_{\alpha (j)} \varphi_{\beta (j) } \right\}
\right) 
\right\} 
\\
\label{WitnessInclusion}
\Phi_j := 
\begin{cases}
\Phi_{j-1} \cup \left\{ \varphi_{\beta \left( j \right)} 
\frac{
x_{
k_{j}
}
}
{
x_{\alpha \left( j \right)}
}
\right\}
& \text{ if } \con \Phi_{j-1} \cup  \left\{ \exists x_{\alpha \left( j \right)} \varphi_{\beta \left( j \right)} \right\}
\\
\Phi_{j-1} & \text{ otherwise}
\end{cases}.
\end{align}
First, the soundness of this algorithm can be checked recursively by showing that each $ X \setminus \free \left(  \Phi_{j}  \right) $ is countable, so that 
in particular the least element in \eqref{FirstNonFreeVar} is taken over a non-empty subset of $\Z^{+}$. This easy step is omitted here.
\\
Secondly, each $\Phi_j$ is easily proven consistent by applying \ref{SingleWitnessInclusion} to first branch of \eqref{WitnessInclusion} each time it is employed. Also, since $\Phi_0 \subseteq \Phi_1 \subseteq \Phi_2 \subseteq \ldots$, each finite subset of $\mathcal{W}_{D} \left( \Phi \right) := \bigcup_{j \in \Z^{+}} \Phi_j$ is a subset of some $\Phi_j$, and thus consistent. This makes $\mathcal{W}_{D} \left( \Phi \right)$ consistent as well. Note that the outcome of the present construction can depend on how the exist statements are initially sorted, i.e.\ on $\alpha$ and $\beta$ (cfr. the parallel situation for $\mathcal{M}^+(\Phi)$ in \ref{MPlusCovering}).
\\
Finally, the last thesis descends from the fact that any consistent superset of $\mathcal{W}_{D} \left( \Phi \right)$ cannot contain any exist-formula whose witness hasn't already been added in \eqref{WitnessInclusion}.
\end{proof}
\section{Putting it all together. Satisfiability and completeness theorems. Adequacy of sequent calculus}
\begin{prop}[Fundamental satisfiability result]
\label{PropSatisfiability}
\FPeval{RuleChoice}{round(2^AssRuleCounter + 2^AntRuleCounter + 2^ReflRuleCounter + 2^SubstRuleCounter + 2^EARuleCounter +2^ESRuleCounter + 2^NorRuleCounter + 2^NorCommRuleCounter + 2^PCtrRuleCounter:0)}
\begin{description}
\mbox{}
\item[Hypothesis]
\mbox{}
\begin{enumerate}
\item
$D \supseteq \FPprint\RuleChoice$
\item
\label{HypConsistency}
$\Phi \subseteq F^S$ is $D$-consistent.
\item
\label{HypFreeCountable}
$ X \setminus \free\left( \Phi \right)$ is countable.
\end{enumerate}
\item[Thesis]
$\Phi$ is satisfiable.
\end{description}
\end{prop}
\begin{proof}
\mbox{}
\begin{enumerate}
\item
$D$ contains enough rules to allow building $\Phi':= \mathcal{M}_D^+ \left( \mathcal{W}_D \left( \Phi \right) \right)$; this is the only step in which hypotheses \ref{HypFreeCountable} and \ref{HypConsistency} are used.
\item
Now $\Phi'$ is consistent because $\mathcal{W}_D$ preserves consistency and thanks to \ref{MPlusConsistency}. It is a covering, too (\ref{MPlusCovering}). So, by definition, it is $D$-maximal and, descending from \ref{LemmaMaximalCovering}, thus both expanded and a minimal covering.
\item
At this point, expandedness permits (see
\ref{LemmaFreeInterpretationConstruction} and \ref{DefPreservation}) the construction of the Henkin interpretation $\mathcal{I}_{\left( \Phi', D  \right)}$.
\item
Finally, $\Phi'$ is witness-furnished thanks to last thesis of \ref{LemmaWitnessInclusion}, so we can apply \ref{LemmaFreeInterpretationModel} and get that $\Phi'$ is satisfied by $\mathcal{I}_{\left( \Phi', D \right)}$.
\end{enumerate}
\mbox{}\\
And so $\Phi$ is.
\end{proof}
It is time, in the end, to see that hypothesis \ref{HypFreeCountable} was a mere technical device to produce the proof of \ref{PropSatisfiability}, which can be restated without it, now that it has been granted.
\begin{teo}[Satisfiability Theorem]
\label{TeoSatisfiability}
\FPeval{RuleChoice}{round(2^AssRuleCounter + 2^AntRuleCounter + 2^ReflRuleCounter + 2^SubstRuleCounter + 2^EARuleCounter +2^ESRuleCounter + 2^NorRuleCounter + 2^NorCommRuleCounter + 2^PCtrRuleCounter:0)}
\begin{description}
\mbox{}
\item[Hypothesis]
\mbox{}
\begin{enumerate}
\item
$D \supseteq \FPprint\RuleChoice$
\item
$\Phi \subseteq F^S$ is $D$-consistent.
\end{enumerate}
\item[Thesis]
$\Phi$ is satisfiable.
\end{description}
\end{teo}
\begin{proof}
\rnote{Da sistemare, ma mi servono definizioni nella sintassi}
\informal{The idea is to cast all the free variables occurring in $\Phi$ into new constants to adjoin to $S$, so to be able to apply \ref{PropSatisfiability}.}
Consider $S':= S \disjcup \left\{ c_1, c_2, \ldots \right\}$. For each $\varphi \in \Phi$ build the formula $\varphi'$ obtained by substituting each occurrence of a free variable $x_j$ with $c_j$. Call $\Phi'$ the set of formulas thus obtained.
\begin{description}
\item[$\Phi'$ is consistent:] 
Take any finite subset $\left\{ \nple{\varphi'}{n} \right\}$, and compare it with the corresponding $\left\{ \nple{\varphi}{n} \right\}$. 
The latter is satisfiable as a set of $S$-formulas by virtue of \ref{PropSatisfiability}, and be $\mathcal{I}$ a model for it. 
Turn $\mathcal{I}$ into a $S'$-interpretation $\mathcal{I}'$ by setting $\mathcal{I}' \left( c_j \right) := \mathcal{I} \left( x_j \right)$. 
Thanks to how we defined $\varphi'$, we have $\mathcal{I}' \left( \varphi' \right) = \mathcal{I} \left( \varphi \right) = \true $, so $\mathcal{I}'$ is a model for $\left\{ \nple{\varphi'}{n} \right\}$. 
Thus the whole $\Phi'$ is consistent.
\end{description}
Let then $\mathcal{I}''$ be a model of $\Phi'$. 
Since $\free \left( \Phi '  \right) = \emptyset$, we can request that $\mathcal{I}'' (x_j) = \mathcal{I}'' (c_j) \ \forall j \in \Z^+$.  
By the same reasoning done for $\mathcal{I}'$, it is shown that $\mathcal{I}'' \left( \varphi' \right) = \mathcal{I}'' \left( \varphi \right)= \true \ \forall \varphi \in \Phi$.
\\
So $\Phi$ is satisfiable.
\end{proof}
\begin{cor}
\FPeval{RuleChoice}{round(2^AssRuleCounter + 2^AntRuleCounter + 2^ReflRuleCounter + 2^SubstRuleCounter + 2^EARuleCounter +2^ESRuleCounter + 2^NorRuleCounter + 2^NorCommRuleCounter + 2^PCtrRuleCounter:0)}
\label{PropCompletenessWRTConsistency}
If $\FPprint\RuleChoice \subseteq D^+ \subseteq D$, then
\begin{align*}
\con_{D^+} \Phi \Leftrightarrow \con_D \Phi
\end{align*}
\end{cor}
\begin{proof}
It suffices to show left-to-right implication.
\\
$\con_{D^+} \Phi$ implies there is a model $\mathcal{I}$ for $\Phi$. Correctness of sequents of $D$ implies (\ref{TheoConsistency}) that $\con_D \Phi$.
\end{proof}
As a major corollary of \ref{TeoSatisfiability}, we now show completeness of first-order calculus. While the former is a constructive result, the latter is highly non-constructive, as it is fundamentally derived from \ref{TeoSatisfiability} using contradiction rules.
\begin{teo}[G\"odel's Completeness Theorem]
\label{TheoCompleteness}
\FPeval{RuleChoice}{round(2^AssRuleCounter + 2^AntRuleCounter + 2^ReflRuleCounter + 2^SubstRuleCounter + 2^EARuleCounter +2^ESRuleCounter + 2^NorRuleCounter + 2^NorCommRuleCounter + 2^PCtrRuleCounter:0)}%
\FPeval{var}{round(2^AssRuleCounter + 2^AntRuleCounter + 2^ReflRuleCounter + 2^SubstRuleCounter + 2^EARuleCounter +2^ESRuleCounter + 2^NorRuleCounter + 2^NorCommRuleCounter + 2^PCtrRuleCounter +2^NCtrRuleCounter :0)}
\begin{description}
\mbox{}
\item[Hypothesis]
$D^+ \supseteq \FPprint\RuleChoice$, $D \supseteq \FPprint\var$,
$\Phi \models \varphi$.
\item[Thesis]
\mbox{}
\begin{enumerate}
\item
If $\varphi$ is negative, then it is provable in $D^+$.
\item
$\varphi$ can be proved in $D$, and there exists a proof of it whose all derivation rules except at most the last are in $D^+$.
\end{enumerate}
\end{description}
\end{teo}
\begin{proof}
\mbox{}
\begin{enumerate}
\item
$\varphi$ is of the form $\lnot \psi$. By contradiction, suppose $\Phi \not \proves_{D^+} \lnot \psi$. By \ref{ConsistencyVSImplication}, $\con_{D^+} \Phi \cup \left\{ \psi \right\}$, thus, using \ref{TeoSatisfiability}, there is a model $\mathcal{I}$ for $\Phi \cup \left\{ \psi \right\}$, so that $\mathcal{I} \left( \psi \right)= \true$. But $\mathcal{I}$ is also a model of $\Phi$, so by hypothesis $\mathcal{I} \left( \lnot \psi \right) = \true$.
\item
$\Phi \models \varphi \Rightarrow \Phi \models \lnot \lnot \varphi$. Applying previous point, we find $\Phi \proves_{D^+} \lnot \lnot \varphi$. Then we use rule \eqref{NotNotRemove} to finish.
\end{enumerate} 
\end{proof}

\appendix

\section{Generalities on semigroups}
\label{SectSemigroups}
Consider a semigroup $\left( M, \circ \right)$.
By virtue of its associativity, the binary operation $\circ$ can be seen also as a function from 
$\bigcup_{j=2}^{+ \infty} { M  }^{j}  $ into $M $. That is, semigroup operation can unambiguously be applied to any $n$-ple, becoming an overloaded function.
When wanting to stress this, functional notation can be employed: rather than using infix notation commonly adopted for semigroup operation: $x_1 \circ x_2$, one writes, for the generic $n$-ple: $\circ \left( \left( x_1, \ldots, x_n \right) \right)$. 
This will be handy in the sequel, besides being more adaptable to the variable argument length just established.
When, on the other hand, normal infix (operational) notation is used, and clarity is safe, the operation symbol $\circ$ will be possibly omitted.
\\
One can push things a little further: consider that in functional notation the associativity is expressed thus:
\begin{align}
\label{FunctionalAssociativity}
\circ\left( \circ \left( \left\{ a \right\} \right) \cartprod \left\{ b \right\} \right) = \circ \left( \left\{ a \right\} \cartprod \left\{ b \right\} \right)
&& \forall a \in \bigcup_{j=2}^{+ \infty} { M  }^{j}, \ b \in \bigcup_{j=1}^{+ \infty} { M  }^{j}.
\end{align}
It then comes natural to extend the function $\circ$ to the whole $\bigcup_{j=1}^{+ \infty} { M  }^{j}$ domain by letting $\circ \left( a \right) = a \ \forall a \in M$, because doing so \eqref{FunctionalAssociativity} results extended as well.\\
So when $\circ$ is used in functional notation, it will be meant to denote this latter extended function.
\begin{defn}
\label{UnambigDef}
A $n$-ple $\left( A_1,\ldots,A_n \right)$ of subsets of $M$ is said unambiguous iff the restriction $\left. \circ \right|_{\prod_{j=1}^n A_{j}}$ is injective.
\end{defn}
\begin{oss}
Definition \ref{UnambigDef} is satisfied in the trivial case of any of the sets $A_{1}, \ldots, A_n$ being empty (any function is injective on empty domain).
\end{oss}
\begin{defn}
A subset $A \subseteq M$ is unambiguous iff $\forall B \subseteq M \, \left( A, B \right)$ is unambiguous; or, equivalently, iff $\left( A, M \right)$ is unambiguous.
\end{defn}
\begin{oss}
\label{UnambigPropagation}
If $\left( A_1, \ldots , A_n \right)$ is a $n$-ple  of not necessarily all distinct subsets of  $M$, each unambiguous, then $\left( \underbrace{A, \ldots, A}_{n \text{ times}}, B \right)$ is unambiguous $\forall B \subseteq M, n \in \Z^+$, which implies that $\circ \left( \prod_{j=1}^n A_j \right)$ is unambiguous.
\end{oss}
\begin{proof}
By induction on $n$.
\end{proof}
\subsection{Unambiguous generators, homomorphic extensions and string substitution}
It is useful to have a set which somehow generates the whole semigroup through its operation; even better if this can be done in only one way.
\begin{defn}
Call $A \subseteq M$ a generator iff $\circ \left|{_{\disjbigcup_{j=1}^{+\infty} A^j}}\right.$ is onto $M$, an unambiguous generator iff is a bijection on $M$.
\end{defn}
\begin{oss}
An unambiguous generator is a generator.
\end{oss}
An unambiguous generator is indeed useful first of all because one can extend a map on it onto the whole semigroup, in the expected manner.
\begin{defn}[Homomorphic extension]
\label{DefHomomorphicExtension}
Given an unambiguous generator $A \subseteq M$ and $f: A \to M$, set 
\begin{align*}
f_{\cartprod}: \disjbigcup_{j=1}^{+ \infty} A^j \ni \left( \nple x n \right)
\mapsto \left( f\left( x_1 \right), \ldots, f\left( x_n \right) \right) \in  \disjbigcup_{j=1}^{+ \infty} M^j  
\end{align*} 
and define the homomorphic extension of $f$ as
\begin{align*}
f_A \left( y \right) := \circ \left( f_{\cartprod} \left( {
\left( \circ \left|_{\disjbigcup_{j=1}^{+\infty} A^j}  \right.  \right)
}^{-1} \left( y \right) \right) \right)
\end{align*}
\end{defn}
It is easy to check that it is defined on the whole $M$ and that it is an extension of $f$.
\informal{
\ref{DefHomomorphicExtension} formally states what becomes the very intuitive concept of substitution as soon as $M$ is the free monoid over the alphabet $A$, which will be the case of our interest in \ref{SectionFormalStructure}: take any string and substitute each occurrence in it of a certain letter with a certain string.
}
In case $f$ is a finite substitution, i.e. there is a finite $F=\left\{ \nple{a}{n} \right\} \subseteq A$ such that $f \left|_{A \setminus F} \right.$ is the identity, which will be the case occurring in the sequel, we will usually employ the alternative notation
\begin{align*}
\varphi
\begin{aligned}
f(a_1) && \ldots && f(a_n)
\\
\hline
a_1 && \ldots && a_n
\end{aligned}
:= f_A \left( \varphi \right)
\end{align*}
\section{Syntax}
\label{SectSyntax}
First-order syntax construction can be found in many references. Here, it is quickly introduced as it will be the material for the whole treatment. The subject is a standard one, so little remarks are added. There are some non-standard variations here: first, only one logical operator symbol is used, namely $\nor$ representing nor binary function (true if and only if both its arguments are false); second, polish notation is adopted; and third, to save words and symbols when having to distinguish between function, relation, and constant symbols of the first-order alphabet, as a mere technical device, we will associate to each symbol a signed-arity ($\#$), whose sign immediately tells about the matter; lastly, only one quantifier is used. These solutions are aimed at simplifying tractation (especially in the proofs of sections ) and making statements shorter, albeit maybe less readable at times.
\subsection{Formal structure}
\label{SectionFormalStructure}
A formal structure is a couple $(S, \#)$ made of an arbitrary, possibly empty, set $S$ (the symbol set) and of a function $\# \colon S \to \Z$ called signed-arity. $s \in S$ is named, respectively:
\begin{itemize}
\item
A constant iff $\#\left( s \right) = 0$
\item
A $\#\left( s \right)$-ary function (or operation) symbol iff $\#\left( s \right) > 0$
\item
A $\left| \#\left( s \right) \right|$-ary relation symbol iff $\#\left( s \right) < 0,$
\end{itemize}
so that $S$ is partitioned into constant, function and relation symbols.
\subsection{Terms and formulas}
\label{TermsAndFormulas}
Take another infinite set $X$ (the variables symbol set) and a further, disjoint set $\left\{ \nor, \equiv , \exists \right\}$, both disjoint with $S$. Call
\begin{align*}
\overline{S} := S \disjcup X \disjcup \left\{ \nor , \equiv , \exists \right\}
\end{align*}
the alphabet, and $S \disjcup X \disjcup \left\{ \equiv  \right\} \subset \overline{S} $ the non-logical alphabet. 
\\
Consider the free semigroup
\begin{align*}
\left( {\overline{S}}^{+} , \circ \right),
\end{align*}
that is, the strings of one or more letters of $\overline{S}$, with string concatenation $\circ$ as a semigroup operation.

Here is shown how to build a couple of subsets of this semigroup in a way such that, when back to semantics, a meaning can be associated to each element of those subsets. 
That is, the strings hereby constructed will have the right structure to be meaningful when each of its characters gets endowed with a `content' and becomes a symbol of it. These strings are indeed exactly all the strings enjoying this feature.
\\
The last preparatory step is the extension of signed-arity function to all the non-logical symbols, by defining
\begin{align*}
\overline{\#} (s) := 
\begin{cases}
\#(s) & s \in S
\\
0 & s \in X
\\
-2 & s \in \left\{ \equiv \right\}
\end{cases}
\end{align*}
\\
Consider finally two auxiliary mappings $\tau, \phi : \powset{\overline{S}^{+}} \to \powset{\overline{S}^{+}}$ acting the following way:
\begin{align*}
\tau \colon Y & \mapsto Y \cup \bigcup_{s \in {\overline{\#}}^{-1} \left( \Z^+ \right)} 
\circ \left( \left\{ s \right\} \cartprod Y^{\overline{\#} \left( s \right) } \right)
\\
\phi \colon Y & \mapsto 
Y \cup \circ \left( \left\{ \nor \right\} \cartprod Y \cartprod Y \right) \cup \circ \left( \left\{ \exists \right\}  \cartprod X \cartprod Y \right)
\end{align*}
Call terms of depth $0$ (or atomic terms) the set $T_0^S := {\overline{\#}} ^ {-1} \left( \left\{ 0 \right\} \right)$, and terms of depth $\leq n \in \Z^+$ the set 
\begin{align*}
T_n^S := \tau^{n} \left( T_{0}^{S} \right);
\end{align*}
here $\tau ^n$ is the $n$-th iterate of the function $\tau$. 
\\
Call terms the union of the terms of any depth:
\begin{align*}
T^S := \bigcup_{j \in \N} T_j^S = \lim_{j \to + \infty} T^S_{j}
\end{align*}
\begin{oss}
\label{TermSymbols}
$\forall n \in \N \  T^S_{n} \subseteq T^S_{n+1} \subseteq T^S \subseteq \left( {\overline{\#}}^{-1} \left( \N \right) \right)^{+}$
\end{oss}
Similarly, call formulas of depth $0$ (or atomic formulas) the set
\begin{align*}
F_0^S := \bigcup_{s \in {\overline{\#}}^{-1} \left( \Z^- \right)} 
\circ \left( \left\{ s \right\} \cartprod \prod_{j = 1}^{\left| \overline{\#}\left( s \right) \right|} T^S \right),
\end{align*}
and formulas of depth $\leq n \in \Z^+$ the image of $F^S_0$ throught the $n$-th iteration of $\phi$:
\begin{align*}
F_n^S := \phi ^ n \left( F_0^S \right).
\end{align*}
Finally, call formulas the union of formulas of any depth
:
\begin{align*}
F^S := \bigcup_{j \in \N} F_j^S = \lim_{j \to + \infty} F^S_{j} .
\end{align*}
\begin{defn}
For concision's sake, we will often use the shortcut character $\lnot$ and write $\lnot \varphi$ instead of $\nor \varphi \varphi$, where $\varphi$ is any formula. One can also consider $\lnot$ as a meta-symbol denoting the mapping $\varphi \mapsto \nor \varphi \varphi$.
\end{defn}
Of course, terms and formulas can be partitioned according to their depth via obvious definitions:
\begin{defn}
\begin{align*}
\overline{T}^S_0 &:= T^S_0 & \overline{F}^S_0 &:= F^S_0 &
\\
\overline{T}^S_n &:= T^S_n \setminus T^S_{n-1} & \overline{F}^S_n &:= F^S_n \setminus F^S_{n-1} ,& n \in \Z^+
\end{align*}
\end{defn}
\begin{oss}
\mbox{}
\begin{itemize}
\item
$\left\{ \overline{T}^S_{n}  \right\}_{n \in \N}$ is a partition of $T^S$.
\item
$\forall n \in \N$
\begin{align}
\overline{T}^S_{n+1} \subseteq \disjbigcup_{s \in {\#}^{-1} \left( \Z^{+} \right)} 
\circ \left( \left\{ s \right\} \cartprod {\left( T^S_n \right)}^{\#\left( s \right)} \right)
& \subseteq T^S_{n+1}
\end{align}
\end{itemize}
\end{oss}
Now it is not hard to guess how the depth of a term or formula will be introduced:
\begin{defn}
\label{DefDepth}
The depth of a term $t \in T^S$ is $\min\left( \left\{ n \st \ t \in T^S_n \right\} \right)$ or, equivalently, the natural number $n$ such that $t \in \overline{T}^S_n$. 
\\
In the same manner, the depth of a formula $\varphi \in F^S$ is $\min\left( \left\{ n \st \ \varphi \in F^S_n \right\} \right)$ or, equivalently, the natural number $n$ such that $\varphi \in \overline{F}^S_n$. 
\end{defn}
\begin{oss}
Alternatively, the depth of a formula can be recursively defined as
\begin{align*}
\operatorname{depth}\left( F^S_0 \right) &:= \left\{ 0 \right\}
\\
\operatorname{depth} \left( \nor \varphi_1 \varphi_2 \right) &:= 
\max \left( \operatorname{depth} \left( \left\{ \varphi_1, \varphi_2 \right\} \right) \right) + 1
\\
\operatorname{depth} \left( \exists x  \varphi \right) &:=  
\operatorname{depth} \left( \varphi \right) + 1,
\end{align*}
and analogously for terms.
\end{oss}
\comment{
\begin{defn}
Since the functions depth of a term and $\overline{\#}$ agree on the meet of their domains, the symbol $\overline{\#}$ can be overloaded to denote both the signed arity of a symbol and the depth of a term, thus extending its domain to $S \cup T^{S} \cup \left\{ \equiv \right\}$. 
\end{defn}
}
\begin{oss}
\begin{align}
\label{TermsFullPartition}
T^S &= T^S_0 \disjcup \disjbigcup_{\substack{s \in {\#}^{-1} \left( \Z^{+} \right) \\ n \in \N}} \left( \circ \left( \left\{ s \right\} \cartprod {\left( T^S_n \right)}^{\#\left( s \right)} \right) \cap \overline{T}^S_{n+1} \right)
\end{align}
\end{oss}
\subsection{Subterms and subformulas}
\begin{lemma}
\label{TermsPartition}
$\forall n \in \N$
\begin{align*}
T^S_{n+1} = T^S_0 \disjcup \bigcup_{s \in {\#}^{-1} \left( \Z^{+} \right)} 
\circ \left( \left\{ s \right\} \cartprod {\left( T^S_n \right)}^{\#\left( s \right)} \right)
\end{align*}
\end{lemma}
\begin{lemma}
\label{TermsUnambig}
$T^S_n$ is unambiguous in $ {\overline{S}}^{+}  \; \forall n \in \N$.
\end{lemma}

\begin{proof}
By induction on $n$. Note that in a free semigroup any subset of elements all of the same length is obviously unambiguous. So $T^S_0$ is unambiguous.
\\
Now suppose $t \extprod[] y = t' \extprod[] y' =: z$ for some $ t, t' \in T^S_{n+1} \; , y, y' \in {\overline{S}}^+$, and to have shown that $T^S_n$ is unambiguous: the goal is to trigger induction by showing $ t=t'$ and $y=y'$.
Since $t, t'$ both have length at least $1$, they both have a first character, and this character is the first character of $z$. Call it $s$. Now the proof is by cases on $s$.
\begin{description}
\item[ $\overline{\#} \left( s \right) = 0$]
By \ref{TermsPartition} this means that both $t , t' \in T^S_0$, which is unambiguous, so $ t=t'$ and $y=y'$.
\item
[$\overline{\#} \left( s \right) \in \Z^+$]
By \ref{TermsPartition} this means that both 
$t, t' \in \circ \left( \left\{ s \right\} \cartprod {\left( T^S_{n} \right)}^{\#\left( s \right)} \right)$, which is unambiguous by \ref{UnambigPropagation}, so $ t=t'$ and $y=y'$.
\end{description}
By \ref{TermSymbols} this parses all the possibilities.
\end{proof}
\begin{teo}
$T^S$ is unambiguous.
\end{teo}
\begin{proof}
Suppose $t, t' \in T^S$ and $y, y' \in \overline{S}^+$ are such that $t y = t' y'$. Call $n$ the greater among the depths of $t$ and $t'$. Since $t, t' \in T^S_n$ and $T^S_n$ is unambiguous, it must be $t=t'$ and $y=y'$.
\end{proof}

Now the mapping from a term to its ordered subterms is introduced.
\begin{defn}
\label{DefSubtermExtractor}
For $t \in T^{S}$ define
\begin{align*}
\mathcal{S}_{S}\left( t \right) := 
\begin{cases}
t & \text{ if } t \in T_{0}^{S}
\\
\left( 
 \circ \left|{_
{
  \left( {\left( T_{\overline{\#}\left( t \right) - 1}^{S} \right)}^{ \overline{\#} \left( t \left( 1:1 \right) \right) } \right) 
}}\right.
\right)^{-1} \left( t \left( 2: \right) \right) & \text{ else}
\end{cases},
\end{align*}
where last assignment is sensible because
\begin{itemize}
\item
As soon as $\overline{\#} \left( t \right) \geq 1$
\begin{itemize}
\item
$t$ has length at least $2$.
\item
\begin{align*}
t\left( 2: \right) \in \circ 
\left(  {\left( T_{\overline{\#}\left( t \right) - 1}^{S} \right)}^{ \overline{\#} \left( t \left( 1:1 \right) \right) } \right)
\end{align*}
\item
$
 \overline{\#} \left( t \left( 1:1 \right) \right) \geq 1
$
\end{itemize}
\item
the restriction is performed over a unambiguous domain
\end{itemize}
\end{defn}
All the results of this section can be adapted and repeated for formulas.
\begin{oss}
\mbox{}
\begin{itemize}
\item
$T^S \cap F^S = \emptyset$
\item
$T^S \disjcup F^S$ is unambiguous.
\end{itemize} 
\end{oss}
\begin{proof}
Note that any string in $T^S \disjcup F^S$ has length $\geq 1$ and that the first character of any term differs from the first character of any formula. This grants both theses. 
\end{proof}
So the operator $\mathcal{S}_S$ can be extended to $T^S \disjcup F^S$. (details todo)
\rnote{
Da introdurre la nozione di occorrenza, di variabile libera, la sostituzione variabile-stringa con precauzioni sulle variabili non libere, il lemma di coincidenza e il lemma di sostituzione} 
\subsection{Families of formulas}
\label{SectFamiliesOfFormulas}
Let $\Phi$ be a generic subset of $F^S$. The following definitions have to do with how many formulas we can fit into $\Phi$, especially in the light of calculus of sequent defined in \ref{TenRules}.
\begin{defn}
$\Phi$ is said to be a covering iff $\forall \varphi \in F^{S}, \ \left\{ \varphi, \nor \varphi \varphi \right\} \cap \Phi \neq \emptyset$
\end{defn}
\begin{defn}
$\Phi$ is said to be patently inconsistent iff $\lnot \left( \Phi \right)  \cap \Phi \neq \emptyset$, that is, iff there is $\varphi \in F^S$ such that $\left\{ \varphi, \nor \varphi \varphi \right\} \subseteq \Phi$. 
\end{defn}
\begin{defn}
$\Phi$ is said to be a minimal covering iff $\forall \varphi \in F^{S}, \ \left\{ \varphi, \nor \varphi \varphi \right\} \cap \Phi \text{ has cardinality } 1 \Leftrightarrow \Phi $ is a covering, while no proper subset of it is $\Leftrightarrow \Phi $ is a covering and not patently inconsistent.
\end{defn}
The two equivalences embedded in the above definition are immediate.

\section{Semantics}
\label{SectSemantics}
\subsection{Interpretations}
Given a formal structure $S$, an interpretation of $S$ is a couple $\mathcal{I} := \left( A, \omega \right)$, of which $A$ (the universe of $\mathcal{I}$) is a set and $\omega \colon S \disjcup X  \to A \disjcup \disjbigcup_{n \in \Z^+} \left( A^{A^n} \disjcup {\truthvalueset}^{A^n} \right)$ satisfies
\begin{align*}
\omega \left( s \right) \in A^{A^n} & \Leftrightarrow \# \left( s \right) = n \in \Z^+
\\
\omega \left( s \right) \in {\truthvalueset}^{A^n} & \Leftrightarrow \# \left( s \right) = -n \in \Z^-
\\
\omega \left( s \right) \in A & \Leftrightarrow \overline{\#} \left( s \right) = 0
\end{align*}
If $\mathcal{I} := \left( A, \omega \right)$ is an interpretation of the formal structure $S$, $\left(A, \left. \omega \right|_{S}\right)$ is said to be an $S$-structure, and $\left(A, \left. \omega \right|_{X} \right)$ an assignment, so that giving an $S$-structure and an assignment both over the same universe equals giving an interpretation of $S$.
\subsubsection{Standard constructions}
\label{StandardConstructions}
\paragraph{Free interpretation}
For any  set $S$ of algebraic symbols, there is a standard interpretation $\left( T^{S} ,  \omega_{S} \right)$ of the subset $S^{\geq} := \#^{-1} \left( \N \right) \subseteq S$ of algebraic symbols, with universe $ T^S = T^{S^{\geq}}$; just take
\begin{align}
\label{DefnFreeInterpretation}
\omega_{S} (s) := 
\begin{cases}
s & \overline{\#} \left( s \right) = 0
\\
\left( \nple{t}{\#\left( s \right)} \right) \mapsto 
s \circ t_{1} \circ \ldots \circ t_{\#\left( s \right)}
& \#(s) > 0
\end{cases}
\end{align}
The only checking to do, namely 
\begin{align*}
\bigcup_{s \in {\#}^{-1} (\Z^{+}) } \omega \left( s \right)\left(  {\left( T^{S} \right)}^{\#\left( s \right)} \right) \subseteq T^{S},
\end{align*}
is easy.
\begin{lemma}
The free interpretation is the only identical interpretation of $S^{\geq}$, in that it satisfies $\mathcal{I}\left( t \right) = t \ \forall t \in T^{S}$.
\end{lemma}
\paragraph{Arbitrary extensions of the free interpretation}
\label{FreeInterprExt}
$(T^{S}, \omega_{S})$ can be extended to an interpretation of the whole $S$ with substantial arbitrarity by choosing $\omega_{S}\left( s \right)$ in ${\truthvalueset}^{{\left( T^{S} \right)}^{\left( -\#\left( s \right) \right)}}$ for each $s \in \#^{-1}\left( \Z^{-} \right)$. That is, given any mapping $f \colon \#^{-1}\left( \Z^{-} \right) \ni s \mapsto f\left( s \right) \in {\truthvalueset}^{{\left( T^{S} \right)}^{\left( -\#\left( s \right) \right)}}$, one can associate to it an interpretation of $(S, \#)$ by taking the free interpretation of $S^{\geq}$ and extending $\omega_{S}$ as just done, to a function we denote with $\omega_S^f$.
Note that giving such a $f$ is equivalent to giving a binary partition of the string set
\begin{align*}
\bigcup_{s \in {\#}^{-1} \left( \Z^- \right)} 
\circ \left( \left\{ s \right\} \cartprod {\left( T^S \right)}^{\left| \# \left( s \right) \right|} \right) \subset F_{0}^{S}
\end{align*}
\paragraph{Quotient interpretation}
Suppose to have an $S$-interpretation $\mathcal{I} = \left( A, \omega \right)$ and an equivalence relation $E$ (whose projector let be denoted by $ p_{E} \colon A \to A/E$) preserved through $\omega$, which means
\begin{gather}
\label{DefPreservation}
\forall s \in {\#}^{-1} \left( \Z \setminus \left\{ 0 \right\} \right)
\\
a_{j} E a_{j}', \   j=1, \ldots, \left| \#\left( s \right) \right| \Rightarrow
\omega \left( s \right) \left( \left( \nple{a}{ \left| \# \left( s \right) \right| } \right) \right) E' \omega \left( s \right) \left( \left( \nple{a'}{ \left| \# \left( s \right) \right| } \right) \right),
\nonumber
\end{gather}
where $E'$ is the equivalence relation on $A \disjcup \truthvalueset$ obtained by setting $E' := E \disjcup \left\{ \left( \false , \false \right) , \left( \true , \true \right) \right\}$.
This property permits building the projected function
\begin{align}
\label{QuotientInterpretation}
&{\frac{\omega}{E}} \left( s \right) :=
\nonumber
\\
&\begin{cases}
p_{E} (\omega(s)) & \overline{\#}(s)=0
\\
\left( p_{E}(a_1), \ldots, p_{E}(a_{\#(s)}) \right) \mapsto 
p_{E} \left( \omega(s) \left( \left( \nple{a}{\#(s)} \right) \right) \right)
& \#\left( s \right) > 0
\\
\left( p_{E}(a_1), \ldots, p_{E}(a_{\#(s)}) \right) \mapsto 
 \omega(s) \left( \left( \nple{a}{\#(s)} \right) \right) 
& \#\left( s \right) < 0
\end{cases}
\end{align}
It is obvious that $\frac{\mathcal{I}}{E} := \left( A/E, \omega / E \right)$ is still an $S$-interpretation.
\paragraph{Assignment mutation}
Denote with $\mathcal{I} \tfrac{\alpha}{}$ the interpretation $\left( A, \overline{\omega} \right)$ obtained from $\mathcal{I} = \left( A, \omega \right)$ by setting{}%
\footnote{Here $\omega$ and $\overline{\omega}$ are considered as relations, which permits seeing them as subsets of the cartesian product of domain with codomain.}
$ \overline{\omega}:= \left( \omega \left|_{S} \right. \right) \disjcup \left\{ \left( x_m , \alpha \left( x_m \right) \right), m \in \Z^+ \right\}$, where $\alpha \in A^X$.
Less formally, we keep the structure of the interpretation while changing the assignment of variables symbols.
The case of interest to us will be that of $\alpha$ equating $\omega \left|_{X}  \right.$ except for a finite subset $\left\{  x_{j_1}, \ldots, x_{j_n} \right\}$ of $X$, in which case we write
\begin{align*}
\mathcal{I}
\begin{aligned}
\alpha(x_{j_1}) && \ldots && \alpha \left( x_{j_n} \right)
\\
\hline
x_{j_1} && \ldots && x_{j_n}
\end{aligned}
:= \mathcal{I} \tfrac{\alpha}{}
\end{align*}
\subsection{Interpretation of non-atomic strings}
An interpretation assigns a meaning to atomic terms and to each operation symbol. Obviously this suffices to recursively assign a meaning (i.e. a value in the universe of the interpretation) to each term by revisiting what done in \ref{TermsAndFormulas} and calculating the value of a term from the value of its (lower-depth) subterms. 
\\
Once done this, one can reiterate an analogous process, and recursively assign a meaning (i.e. a value in $\truthvalueset$) to each formula by regarding it as made of boolean symbols and subformulas, and ultimately of boolean symbols and atomic formulas, and in turn the latter as made of relation symbols, $\equiv$'s and terms.
\\
Formally, this is done by extending $\omega$ of the given interpretation $\mathcal{I} := \left( A, \omega \right)$ first to all $T^S$ and then to all $F^S$, as follows. 
Before formalization, it should be noted that it is customary to abuse a bit the notation and use the interpretation symbol $\mathcal{I}$ altogether in place of the proper function symbol $\omega$, so one can write statements like, e.g., $\mathcal{I}\left( t \right) \in A$ and $\mathcal{I}\left( \varphi \right) = \true$.
\begin{defn}[Recursive interpretation of terms]
Define, for $t, \nple{t}{n} \in T_0^S$:
\begin{align*}
\mathcal{I}\left( t \right) := \omega \left( t \right) && 
\mathcal{I} \left( \left( \nple{t}{n} \right) \right):= \left( \mathcal{I}\left( t_1 \right), \ldots, \mathcal{I}\left( t_n \right) \right).
\end{align*}
Then recursively%
\footnote{See \ref{DefSubtermExtractor} for the definition of $\mathcal{S}$. Basically, it returns all subterms and subformulas of a given term/formula into a $n$-ple.}:
\begin{align*}
\mathcal{I}\left( t \right) := \omega \left( t\left( 1:1 \right) \right)
\left( \mathcal{I} \left( \mathcal{S} \left( t  \right) \right) \right)
&&
\mathcal{I} \left( \left( \nple{t}{n} \right) \right):= \left( \mathcal{I}\left( t_1 \right), \ldots, \mathcal{I}\left( t_n \right) \right).
\end{align*}
\end{defn}
\begin{defn}[Interpretation of atomic formulas]
\begin{align*}
\mathcal{I} \left( \varphi \right) :=
\begin{cases}
\omega \left( \varphi \left( 1:1 \right) \right) \left( \mathcal{I} \left( \mathcal{S} \left( \varphi \right) \right) \right)
& \text{ if } \varphi\left( 1:1 \right) \neq \equiv
\\
\true & \text{ if } \varphi = \equiv t_1 t_2 \text{ and } \mathcal{I}\left( t_1 \right) = \mathcal{I} \left( t_2 \right)
\\
\false & \text{ otherwise}
\end{cases}
\end{align*}
\end{defn}
\begin{defn}[Recursive interpretation of non-atomic formulas]
\begin{align*}
\mathcal{I} \left( \nor \varphi_1 \varphi_2 \right) := \true && 
\Leftrightarrow 
&& \mathcal{I} (\varphi_1) = \false \text{ and } \mathcal{I} (\varphi_2) = \false
\\
\mathcal{I} \left( \exists x_m \varphi \right) := \true && 
\Leftrightarrow && 
\mathcal{I} \tfrac{a}{x_m} (\varphi) = \true \text{ for some } a \in A
\end{align*}
\end{defn}
\begin{defn}[Satisfaction and model]
An interpretation $\mathcal{I}$ satisfies $\Phi \subseteq F^S$, otherwise said that $\mathcal{I}$ is a model of $\Phi$, iff $\mathcal{I} \left( \Phi \right)= \left\{ \true \right\}$. This is expressed in symbols as $ \mathcal{I} \models \Phi$. 
\end{defn}

\section{Table of notations and meta-symbols}

\begin{tabular}[c]{lr}
\hline
$\disjcup$ & disjoint union
\\
$\cartprod$ & cartesian product of two sets (infix)
\\
$\prod_{j=1}^n A_j$ & cartesian product of $n$ sets%
\footnotemark
\\
$A^n$ & shortcut for  $\prod_{j=1}^n A$
\\
$A^B$ & the family of functions mapping set $B$ into set $A$
\\
$\left. f \right|_A$ & restriction of $f$ to subdomain $A$
\\
$\powset{A}$ & the power set of $A$
\\
$s\left( 1:1 \right)$ & the first character of the string $s$
\\
$s\left( 2: \right)$ & the unique string such that $ s(1:1) s(2:) = s$
\\
$ \N $ & the non-negative integers 
\\
$\true, \false$ & semantic values of true and false for classical logic
\\
$S^+$ & the free semigroup generated by $S$,
\\
& i.e. the non-empty finite strings over the alphabet $S$
\\
\hline
\end{tabular}
\footnotetext{The two writings $\left( \prod_{j=1}^m A_j \right) \cartprod 
\left( \prod_{j=1}^n A_{m+j} \right)$ and $\prod_{j=1}^{m+n} A_j$ will denote the \emph{same} set. That is, cartesian product is considered associative, although formally this is a slight formal abuse. On the other hand, here the $A_j$'s will always be such as not to pose problems with this abuse.\\
As a convention,  $\prod_{j=1}^1 A_j$ is $A_1$ itself.
}

\rnote{Da rifare con bibtex}

\end{document}